\documentclass[a4 paper, 10pt]{article}
\usepackage{hyperref}
\usepackage{amssymb,dsfont,amsthm,amsmath,makeidx,verbatim,latexsym,amsfonts,mathrsfs}
\usepackage{cancel}
\usepackage{graphicx}
\usepackage[none]{hyphenat}
 \usepackage{hyperref}
\usepackage[english]{babel}
\addto{\captionsenglish}{%

}
\DeclareMathAlphabet{\mathpzc}{OT1}{pzc}{m}{it}

\begin{document}
\hfuzz5pt
\theoremstyle{plain}
\newtheorem{theorem}{\textbf{Theorem}}[section]
\newtheorem{lemma}[theorem]{\textbf{Lemma}}
\newtheorem{proposition}[theorem]{\textbf{Proposition}}
\newtheorem{corollary}[theorem]{\textbf{Corollary}}
\newtheorem{claim}[theorem]{\textbf{Claim}}
\newtheorem{addendum}[theorem]{\textbf{Addendum}}
\newtheorem{definition}[theorem]{\textbf{Definition}}
\newtheorem{remark}[theorem]{\textbf{Remark}}
\newtheorem{example}[theorem]{\textbf{Example}}
\newtheorem{conjecture}[theorem]{\textbf{Conjecture}}
\newtheorem{notation}[theorem]{\textbf{Notation}}
\renewcommand{\baselinestretch}{1.50} 

\pagenumbering{arabic} \baselineskip 10pt
\newcommand{\disp}{\displaystyle}
\thispagestyle{empty}

\title{On Tempered Ultradistributions in Classical Sobolev Spaces}
	
\author{A. U. Amaonyeiro$^{1}$ and M.E. Egwe$^{2*}$ \\Department of Mathematics, University, Ibadan, Nigeria\\anslemamaonyeiro@uam.edu.ng\;\;
murphy.egwe@ui.edu.ng\\ * Corresponding author.}	
\maketitle		

\noindent\rule{\textwidth}{1pt}


\begin{abstract}{\noindent  We construct and investigate  the properties of tempered ultradistribution spaces in Sobolev spaces. A new Sobolev space preserving the original properties and condition whose derivatives are linear continuous operators embedding in $L^p$ for $1\leq p\leq \infty$ is characterized. Moreover, we also consider some Sobolev embedding theorems involving rapidly decreasing functions, and finally, we prove the extension of Rellich's compactness theorem.}\end{abstract}
\ \\
\textbf{keywords:}Tempered ultradistributions, Classical Sobolev spaces, Embedding, Compactness\\


\section{Introduction}\label{sec1}
\noindent We conduct a survey on the extension of a special type of generalized functions called the class of tempered ultradistributions in the classical setting of Sobolev spaces.\ \\
The notion of classical Sobolev spaces $\mathcal{W}^{l,p}$ was introduced by S. L. Sobolev in the 1930s basically developed for the framework of modern theory of partial differential equations. The development of these spaces has a foundation in the theory of generalized functions. For details about the spaces, refer to \cite{friedman}, \cite{narici}  and \cite{leoni}. This theory was facilitated by the existence of embedding between many of the function spaces like the $L^p$-spaces.\ \\
\noindent In \cite{leoni} one sees the space $\mathcal{W}^{l,p}$ with constant exponent $p$ as the space of all functions in $L^p$ whose distributional partial derivatives belong to $L^p$, that is, for all $j\in\mathbb{N}^n$ there are $\psi_{j}\in L^{p}$ such that
\[\int_{\Omega}\mu\frac{\partial^{\alpha}\eta}{\partial y_{j}^{\alpha}}dy=(-1)^{\vert\alpha\vert}\int_{\Omega}\psi_{j}\eta dy\quad\forall\quad \eta\in C_{c}^{\infty}\quad\text{for}\quad p\in [1,\infty]\]
In this regard, $\psi_j$ is viewed as the weak derivative or the distributional derivative of $\mu\in L^p$. The work on Sobolev space was extended using the defining functions as tempered distributions in \cite{pahk}. In the work, it is defined that a tempered distribution $f\in \mathcal{S}$ belongs to the Sobolev space $\mathcal{W}^{l,p}\equiv \mathcal{W}^{l,p}(\mathbb{R}^{n})$, $l\in\mathbb{R}$ if $\hat{\mu}(x)$ is a function such that the following condition is satisfied:
\begin{equation}
\label{eqnintro}
\int\vert\hat{\mu}(x)\vert^{2}(1+\vert x\vert^{2})^{l}dx<\infty
\end{equation}
The notion of generalized functions bridge the relationship between many function spaces and harmonic analysis. For instance, this theory provides the preservation of properties of special class of ultradistributions named the tempered ultradistributions and some related function space. For details on tempered ultradistributions, see references Friedman \cite{friedman}, Silva \cite{silva} and Roumieu \cite{hormander}.\ \\
\noindent This paper is concerned with the inclusion of tempered ultradistributions in the classical setting of Sobolev spaces and see the behavioral properties of the underlying functions. The aim can be viewed as taking fundamental Sobolev spaces of infinitely ultradifferentiable tempered ultradistributions with slow growth conditions whose derivatives belong to the $L^p$-spaces.\ \\
\noindent This paper is organized as follows: section 2 gives some detailed information about the classical Sobolev spaces and tempered ultradistributions, and present some auxiliary statements which play significant roles in the proof of the main results. In section 3, we present the inclusion of tempered ultradistributions in Sobolev spaces with their underlying structure and properties giving emphasis on the exponent. Finally, in section 4, some main results are presented in form of embeddings between the function spaces under considerations, and we present the extension of the Rellich's compactness theorem involving the newly defined classical Sobolev spaces.

\section{Preliminaries}
\subsection{On Classical Sobolev Space}
Throughout this paper, let $\Omega$ denotes a non-empty open domain $n$-dimensional space basically for the differentiability and integrability of functions. We assume that $\Omega$ is an open subset of $\mathbb{C}^n$, the complex plane. We introduce some of the notations to be used in this work. Basically we outline below some concepts of the theory of topological vector spaces with its important roles in the study of temperd ultradistributions and Sobolev spaces. For detailed discussion of these topics, refer to \cite{leoni} and \cite{silva}.
\begin{notation}
Denote by $\text{supp}(\mu)$ the support of functions $\mu$; $\mathcal{D}(\Omega)$ or $C_{0}^{\infty}$ the space of all $C^\infty$-functions with compact support in $\Omega$; the function spaces $C^{l}(\Omega)$, $C^{l}(\bar{\Omega})$ and $C_{0}^{l}$ of all functions with derivatives of order $l$; and the classes $C^{l,\beta}(\Omega)$, $C^{l,\beta}(\bar{\Omega})$ and $C_{0}^{l,\beta}(\Omega)$ of all functions with derivatives of order $l$ satisfying H$\ddot{\text{o}}$lder inequality of exponent $\beta\in (0,1]$.\ \\
Let $\beta\in\mathbb{N}^n$ be a multi-index such that $\displaystyle\vert\beta\vert=\beta_{1}+\cdots+\beta_n$, $\beta!=(\beta_{1}!,\cdots,\beta_{n}!)$, $D^{\beta}=D_{\xi_{1}}^{\beta_1},\cdots,D_{\xi_{n}}^{\beta_n}$, where $\displaystyle D_{\xi_{i}}=\frac{\partial}{\partial\xi_{i}}$, and $\displaystyle \xi^{\beta}=\Big(\xi_{i},\cdots,\xi_{n}\Big)^{(\beta_{1},\cdots,\beta_{n})}=\xi^{\beta_{1}},\cdots,\xi^{\beta_n}$.
\end{notation}

\begin{definition}
\label{defdistri}
We define the space $\mathcal{D}'(\Omega)$ as a dual space of $\mathcal{D}$ consisting of distributions according to L. Schwartz \cite{schwartz}. The space, denoted by $L^{p}_{\text{loc}}(\Omega)$, is defined as a space of all $p$th locally integrable functions $\mu$ such that the following condition is satisfied:
\[\int_{\Omega}\vert\mu\vert^{p}d\xi<\infty.\]
\end{definition}
There are functions $\eta\in C^{1}(\mathbb{R}^{n})\setminus C^{2}(\mathbb{R}^{n})$. We present the $\beta$th partial derivative of $\mu\in L_{\text{loc}}^{1}(\Omega)$ as defined below:

\begin{definition}
\label{defpartialderiv}
Let $\mu\in L_{\text{loc}}^{1}(\Omega)$ and let $\beta\in\mathbb{N}^n$ be a multi-index. We define $\eta\in L_{\text{loc}}^{1}(\Omega)$ as the $\beta$the partial derivative of $\mu$, denoted $D^{\beta}\mu=\eta$, if the following condition is satisfied:
\[\int_{\Omega}\mu D^{\beta}\psi d\xi=(-1)^{\vert\beta\vert}\int_{\Omega}\eta\psi d\xi\quad\forall\quad \psi\in C_{0}^{\infty}(\Omega)\]
\end{definition}
From definition \ref{defpartialderiv}, we observed that for $\beta=0=(0,\cdots,0)$, we have $D^{0}\mu=D^{(0,\cdots,0)}\mu=\mu$. Similarly then we have that
\[D_{i}\mu=\frac{\partial\mu}{\partial\xi_{i}}=D^{(0,\cdots,1,\cdots,0)}\mu,\quad\text{for}\quad i\in\mathbb{N}^{n}\]
In particular $D\mu$ is termed the weak gradient of $\mu$.

\begin{remark}
\label{rema.e}
Weak derivatives are basically functions satisfying the condition of integration by parts, and change of of any function on a measure of zero has no effect on its weak derivatives. Recall that weak derivatives generalize the classical derivative, and that $D^{\beta}\mu=0$ almost everywhere also implies that $\mu\equiv 0$ almost everywhere.
\end{remark}
The following lemma will back up our claim in Remark \ref{rema.e}.

\begin{lemma} \cite{adams}
\label{lemweakderivative}
A weak $\alpha$th partial derivative of $\mu$, if it exists, is uniquely defined up to a set of measure zero.
\end{lemma}

We now present the definition of classical Sobolev spaces below.

\begin{definition}
\label{def1}
Let $\Omega$ be an open subset of $\mathbb{R}^n$, and let $\beta$ be a multi-index such that $\vert\beta\vert\leq l$ for $l\in\mathbb{R}$. We define the classical Sobolev space, denoted  $\mathcal{W}^{l,p}(\Omega)$, as a function space consisting of functions $\mu\in L^{p}(\Omega)$ such that $D^{\beta}\mu\in L^{p}(\Omega)$ provided the weak derivative $D^{\beta}\mu$ exists. That is, we have
\begin{equation}
\label{eqn1.0}
\mathcal{W}^{l,p}(\Omega)\equiv\Big\{\mu\in L^{p}(\Omega):D^{\beta}\mu\in L^{p}(\Omega),\quad\vert\beta\vert\leq l\Big\}
\end{equation}
\end{definition}
The defining norm for $\mathcal{W}^{l,p}$ is given by
\begin{equation}
\label{eqn1.1}
\Vert\mu\Vert_{\mathcal{W}^{l,p}(\Omega)}=\Big(\sum_{\vert\beta\vert\leq l}\int_{\Omega}\vert D^{\beta}\mu(\xi)\vert^{p}d\xi\Big)^{\frac{1}{p}},\quad \mu\in \mathcal{W}^{l,p}(\Omega),\quad 1\leq p<\infty
\end{equation}
and for $p=\infty$, its norm is given as
\begin{equation}
\label{eqn1.2}
\Vert\mu\Vert_{\mathcal{W}^{l,\infty}(\Omega)}=\sum_{\vert\beta\vert\leq l}\Vert D^{\beta}\mu\Vert_{L^{\infty}(\Omega)}=\sum_{\vert\beta\vert\leq k}\underset{\Omega}{\text{esssup}}\vert D^{\beta}\mu\vert\quad \mu\in \mathcal{W}^{l,\infty}(\Omega)
\end{equation}
where  we can define the norm as $\displaystyle \Vert D^{\beta}\mu\Vert_{L^{\infty}(\Omega)}=\underset{\Omega}{\text{esssup}}\vert D^{\beta}\mu\vert$.\\
In summary the Sobolev space $\mathcal{W}^{l,p}(\Omega)$ is a function space consisting of functions that possesses some properties of distributional derivatives up to order $l$ and belongs to $L^p$-space. The functions in $\mathcal{W}^{l,p}$ are equal almost everywhere in $L^{p}$. We identify different ways of describing norms for the classical Sobolev spaces. For more details in different norms on $\mathcal{W}^{l,p}$ for $1\leq p<\infty$ and $l=1,2,\cdots$ etc we refer, for instance to \cite{leoni}.
\begin{example}
\label{exunbounded}
Given $n\geq 2$ and $\mu:B(0,1)\longrightarrow\mathbb{R}^{+}=[0,\infty]$ and $\mu(\xi)=\vert \xi\vert^{-\beta},\quad \beta>0$. Then we have that $\mu$ is not bounded within any neighborhood of the origin. That is, $\mu\in C^{\infty}(B(0,1)\setminus\lbrace 0\rbrace)$, and it also means that $\mu$ is not $p$th locally integrable. This also implies that some functions in $\mathcal{W}^{1,p},\quad 1\leq p<n,\quad n\geq 2$, may not be bounded in every open subset.
\end{example}

\begin{example}
However the same function defined in Example \ref{exunbounded} does not belong to $\mathcal{W}^{1,n}\Big(B(0,1)\Big)$, and this brings some unbounded functions in $\mathcal{W}^{1,n}$ for $n\geq 2$. However, we define the following function $\mu:B(0,1)\longrightarrow\mathbb{R}$ by
\[\mu(\xi)=\begin{cases}
\log\Big(\log(1+\frac{1}{\vert \xi\vert})\Big), & \xi\neq 0\\
   0, & \xi=0
\end{cases}
\]
It shows that $\mu\in \mathcal{W}^{1,n}\Big(B(0,1)\Big)$ when $n\geq 2$, but $\mu\notin L^{\infty}\Big(B(0,1)\Big)$.
\end{example}

The space $L^\infty$ can be used to construct a function in $\mathcal{W}^{1,n}\Big(B(0,1)\Big)$ that is unbounded in every open subset of $B(0,1)$, that is, functions unbounded within the neighborhood of zero. Thus functions defined in $\mathcal{W}^{1,p},\quad 1\leq p\leq n,\quad n\geq 2$ are not continuous. Every function defined in $\mathcal{W}^{1,p}$ with $p>n$ coincide with a continuous function almost everywhere. Also the function $\eta(0,1)\longrightarrow\mathbb{R}$, defined by
\[\eta(\xi)=\eta(\xi_{1},\cdots,\xi_{n})=\begin{cases}
1, & \xi_{n}\geq 0\\
0,  & \xi_{n}<0
\end{cases}
\]
is a member of $\mathcal{W}^{1,p}\Big(B(0,1)\Big)$ for any $1\leq p\leq \infty$.

\subsection{On Some Properties of Sobolev Spaces}
We assume that the domain under consideration is smooth. In this section we are more concerned with the integral representation of functions in the classical Sobolev space $\mathcal{W}^{l,p}$ which vanish on the boundary of $\Omega$, and also this representation entails the properties of the space.  We refer for instance to \cite{leoni} for the properties of classical Sobolev spaces. \ \\
Thus we present general properties of the classical Sobolev spaces involving weak derivatives below.
\begin{lemma}\cite{adams}
\label{lem1.1}
Assume that $\mu,\nu\in \mathcal{W}^{l,p}(\Omega)$ and $1\leq \vert\beta\vert\leq l$. Then we have the following:
\begin{itemize}
\item[(i)] $\displaystyle D^{\beta}\mu\in \mathcal{W}^{l-\vert\beta\vert,p}(\Omega)$
\item[(ii)] $\displaystyle D^{\beta}(D^{\alpha}\mu)=D^{\alpha}(D^{\beta}\mu)$ for all multi-indices $\alpha,\beta$ with $\vert\beta\vert\leq l$
\item[(iii)] $\displaystyle \forall\quad\lambda,t\in \mathbb{R},\quad \lambda\mu+t\nu\in \mathcal{W}^{l,p}(\Omega)$ and
\[D^{\beta}(\lambda\mu+t\nu)=\lambda D^{\beta}\mu+t D^{\beta}\nu,\]
\item[(iv)] If $\Omega'\subset \Omega$ is open, then $\mu\in \mathcal{W}^{l,p}(\Omega')$,
\item[(v)] If $\eta\in C^{\infty}_{0}(\Omega)$, then $\displaystyle \eta,\mu\in \mathcal{W}^{l,p}(\Omega)$ and
\[D^{\alpha}(\eta\mu)=\sum_{\beta\leq \alpha}\binom{\alpha}{\beta}D^{\beta}\eta D^{\alpha-\beta}\mu\]
where
\[\binom{\alpha}{\beta}=\frac{\alpha!}{\beta!(\alpha-\beta!)},\quad \alpha!=\alpha_{1}!,\cdots,\alpha_{n}!\]
and $\beta\leq \alpha$ means that $\beta_{i}\leq \alpha_{i}$ $\forall\quad i=1,2,\cdots,n$.
\end{itemize}
\end{lemma}
\begin{remark}
The condition (v) of Lemma \ref{lem1.1} is called Leibnitz formula. Hence theories of weak derivatives and classical derivatives have the same properties.
\end{remark}
For further properties of the classical Sobolev spaces we refer, for instance to  \cite{leoni} and \cite{triebel}.\ \\
We now introduce a somewhat different class of generalized functions called tempered ultradistributions.

\subsection{On The Spaces $\mathcal{U}$ and $\mathcal{U}'$}
We introduce a special class of generalized functions called tempered ultradistributions defined on the spaces of rapidly decreasing ultradifferentiable functions. We denote by $\mathcal{U}'$ the strong dual of the space $\mathcal{U}$ and call its Fourier distributions of exponential type or tempered ultradistributions as seen by Silva \cite{silva}.\\
We introduce the following definitions.
\begin{definition}
\label{defultra}
We define the space $\mathcal{U}(\mathbb{C}^{n})$ as the function space of all rapidly decreasing ultradifferentiable functions $\mu\in\mathcal{U}\subset\mathcal{S}$ such that the condition below is satisfied:
\begin{equation}
\label{eqn1}
\Vert \mu\Vert_{p}=\sup_{\vert\text{Im}(\xi)\vert<p; \xi\in\mathbb{C}}\Big\{(1+\vert \xi\vert^{p})^{l}\vert \mu(\xi)\vert\Big\}<\infty\quad\forall\quad p\in\mathbb{N}\quad\text{and}\quad \vert\beta\vert\leq l,\quad l\in\mathbb{R}.
\end{equation}
\end{definition}
\ \\
The dual of the space $\mathcal{U}$ is the space of all tempered ultradistributions, denoted by $\mathcal{U}'$.
The continuous linear functional $\mu:\mathcal{U}\longrightarrow\mathbb{C}^{n}$ can be described as the Fourier transform of exponential distributions or as finite order derivatives of bounded functions with exponential growth. The spaces of tempered ultradistributions can be represented by analytical functions. Next we present the vanishing condition and compact supported tempered ultradistributions.

\begin{definition}
\label{defvanish}
A tempered ultradistribution $\mu\in\mathcal{U}'$ is said to vanish in an open subset $\Omega$ if $\displaystyle (1+\vert\xi\vert^{p})^{l}\mu(x+iy)-(1+\vert\xi\vert^{p})^{l}\mu(x-iy)\longrightarrow 0$ as $\vert\xi\vert\to 0$ and $y\to 0$.
\end{definition}

\begin{definition}
\label{defcompact}
A tempered ultradistribution $\mu\in\mathcal{U}'$ has a compact support if there is a representative function $\hat{\mu}$ vanishing at $\infty$. Equivalently, $\mu$ vanishes outside a compact support.
\end{definition}
For more details about the underlying topology, properties about the spaces $\mathcal{U}$, $\mathcal{U}'$, $\mathcal{D}$, $\mathcal{S}$ and $\mathcal{E}$ we refer, for instance to \cite{silva} and \cite{roumieu}.

\section{On Inclusion of Tempered Ultradistribution in Classical Setting of Sobolev Spaces}
We include tempered ultradistributions in the classical setting of Sobolev space framework. Some of the basic properties of both spaces have been introduced in the previous sections. We discuss when the two spaces of respective underlying defining functions agree. Intuitively tempered ultradistributions spaces agree with the classical Sobolev spaces if their respective linear functionals are continuous and if the differentiable or ultradifferentiable functions are dense in the spaces. In particular this property reflects the fact the continuity of the respective linear functionals defined on some different spaces of test functions or rapidly decreasing ultradifferentiable functions is built into the definition of the classical Sobolev space combined with the density of the defining functions. \ \\
We define the class of tempered ultradistributions in Sobolev space.

\begin{definition}
\label{def**}
Let $\Omega\subset \mathbb{C}^n$ be an open subset. Then the Sobolev space $W_{\mathcal{U}}^{l,p}(\Omega)$ is the space of all tempered ultradistributions $\mu\in\mathcal{U}'$ whose ultradistributional (generalized) gradient $\nabla\mu$ belongs to $L^{p}(\Omega;\mathbb{C}^{n})$ for $\vert\beta\vert\leq l$.
\end{definition}

We present the similar definition below
\begin{definition}
\label{def2.1}
Let $p\in[1,\infty)$, $l\geq 0$ integer and $\beta$ multi-index be given. Then the tempered ultradistribution $\mu\in\mathcal{U}'(\mathbb{C})$ belongs to $\mathcal{W}^{l,p}(\Omega)$, denoted $\mathcal{W}^{l,p}_{\mathcal{U}}$, if $\hat{\mu}$ is a function in $\mathcal{D}'$ and
\begin{equation}
\label{eqn2.1}
\Vert \mu\Vert_{\mathcal{W}^{l,p}_{\mathcal{U}}}=\Bigg[\int_{\mathbb{C}^{n}}(1+\vert\xi\vert^{p})^{l}\vert\hat{\mu}(\xi)\vert^{p}d\xi\Bigg]^{\frac{1}{p}}<\infty\quad\text{for}\quad\vert\beta\vert\leq l
\end{equation}
Similarly we have
\begin{equation}
\label{eqn2.1**}
\Vert \mu\Vert_{\mathcal{W}^{l,p}_{\mathcal{U}}}=\Bigg[\int_{\mathbb{C}^{n}}(1+\vert\xi\vert^{p})^{l}\vert\partial^{\beta}\hat{\mu}(\xi)\vert^{p}d\xi\Bigg]^{\frac{1}{p}}<\infty\quad\text{for}\quad\vert\beta\vert\leq l
\end{equation}

The classical tempered space $\displaystyle \mathcal{W}_{\mathcal{U}}^{l,p}(\Omega)$ can be represented by:
\[\mathcal{W}_{\mathcal{U}}^{l,p}(\Omega):=\Big\{\mu\in \mathcal{U}':(1+\vert\xi\vert^{p})^{l}D^{\beta}\mu\in L^{p}\quad\text{for}\quad\vert\beta\vert\leq l\Big\}\]
where $\mathcal{D}'$ is the space of distributions.
\end{definition}
We observed that the space $\mathcal{S}$ (the space of all tempered distributions) is contained in $\displaystyle \mathcal{W}_{\mathcal{U}}^{l,p}(\Omega)$. Then it follows that the space $\mathcal{U}$ is contained in $\displaystyle \mathcal{W}_{\mathcal{U}}^{l,p}(\Omega)$ for $\mathcal{U}\subset\mathcal{S}$. For $l=0$, then $\displaystyle \mathcal{W}_{\mathcal{U}}^{0,p}$ coincides with classical $L^{p}$.\ \\
We see that $\displaystyle \mathcal{W}_{\mathcal{U}}^{l,p}(\Omega)$ has a structure of Hilbert space as presented in Theorem \ref{thm2.4}.

\begin{theorem}
\label{thm2.4}
The space classical tempered $\displaystyle \mathcal{W}_{\mathcal{U}}^{l,p}(\Omega)$ exhibit the structure of Hilbert space with inner product defined by
\begin{equation}
\label{eqn2.1}
\langle\mu,\varphi\rangle_{l}=\int_{\Omega}(1+\vert\xi\vert^{p})^{l}\hat{\mu}(\xi)\overline{\hat{\nu}(\xi)}d\xi
\end{equation}
\end{theorem}

\begin{proof}
It is trivial that the defined inner product is an inner product on $\displaystyle \mathcal{W}_{\mathcal{U}}^{l,p}$ and since $\displaystyle L^{2}\Big(\mathbb{C}^{n},(1+\vert\xi\vert^{p})^{l}d\xi\Big)$ is complete and the Fourier transform of distributions of exponential type is a holomorphism of $\mathcal{U}'$, it follows that $\mathcal{W}_{\mathcal{U}}^{l,p}$ is Banach.
\end{proof}

We note that every ultradifferential operator $P$ with constant coefficient maps $\mathcal{W}^{l,p}_{\mathcal{U}}$ to $W^{l-\vert\beta\vert,p}_{\mathcal{U}}$ if the degree of the polynomial is less than or equal to $\vert\beta\vert$. We present the continuity of the ultradifferential operator in the classical Sobolev space as reflected in the result below:

\begin{theorem}
\label{thm2.7}
Let $\displaystyle P(D)=\sum_{\vert\beta\vert\leq l}b_{\beta}D^{\beta}$ be an ultradifferential operator of infinite order such that there are constants $M>0$ and $H>0$ for which the following condition holds:
\begin{equation}
\label{eqn2.2}
M=\sup\Bigg(\frac{\beta!\vert b_{\beta}\vert} {\Big(\frac{H}{\sqrt{H}}\Big)^{\vert \beta\vert}}\Bigg)\quad\forall\quad \beta\quad\text{and}\quad \vert\beta\vert\leq l
\end{equation}
If $\displaystyle\mu \in  \mathcal{W}_{\mathcal{U}}^{l,p}$ then $P$ belongs to $\mathcal{W}^{l-t,p}$. \\
Also, the map $\displaystyle P:\mathcal{W}^{l,p}_{\mathcal{U}}\longrightarrow \mathcal{W}^{l-t,p}_{\mathcal{U}}$ is continuous where $t$ represents the degree of the polynomial $P$.
\end{theorem}

\begin{proof}
We prove that $\displaystyle \vert b_{\beta}\vert\leq M\frac{\Big(\frac{H}{\sqrt{H}}\Big)^{\vert \beta\vert}}{\beta!}\quad\forall\quad \beta\quad\text{and}\quad \vert\beta\vert\leq l$, and the necessary and sufficient condition of (\ref{eqn2.2}) for the proof of the result to hold is only when
\[\vert P(\xi)\vert\leq M(1+\vert\xi\vert^{p})^{l}\]
for all $\xi\in\mathbb{C}^n$. Hence the result follows.
\end{proof}

We give an example of an ultradifferential operator of infinite order of rapid decrease as contained below.

\begin{lemma}
\label{lem2.8}
Let $l$ be a non-negative integer, $\xi\in\mathbb{C}^n$, and let $(1+\vert\xi\vert^{p})^{l}$ be an ultrapolynomial of degree $l$, for $1\leq p<\infty$. The map $\displaystyle \varphi\mapsto (1+\vert\xi\vert^{p})^{l}\vert\varphi\vert$ is an isomorphism from the space $\mathcal{W}^{l,p}_{\mathcal{U}}$ onto $\mathcal{W}_{\mathcal{U}}^{l-\vert\beta\vert,p}$ for $\vert\beta\vert\leq l$.
\end{lemma}

\begin{proof}
Let $\varphi\in \mathcal{W}_{\mathcal{U}}^{l,p}$ such that $\vert\beta\vert\leq l$ and for $\leq p<\infty$, we get
\begin{equation}
\label{eqn3.1}
\Big\vert (1+\vert\xi\vert^{p})^{l}((1+\vert\xi\vert^{p})^{-l}D)^{\beta_{1}}(1+\vert\xi\vert^{p})^{t-(u+m)-1}\Big((1+\vert\xi\vert^{p})^{m}\varphi(\xi)\Big)\Big\vert=\Big\vert (1+\vert\xi\vert^{p})^{l}((1+\vert\xi\vert^{p})^{-1}D)^{t-u-m}\varphi(\xi)\Big\vert
\end{equation}
Since $\displaystyle \sup_{\xi\in\mathbb{C}^{n}}\Big\{  (1+\vert\xi\vert^{p})^{l} \vert\varphi(\xi)\vert\Big\}<\infty$ then multiplying (\ref{eqn3.1}) by the operator $[(B+\beta)b_{i}]^{-1}$ with supremum $\forall$ $\xi\in (0,n)$ for $n>1$, we obtain
\begin{equation}
\label{eqn3.2}
\partial^{t,u+m}_{l,\alpha}\Big[(1+\vert\xi\vert^{p})^{l}\varphi(\xi)\Big]=\partial^{t,u}_{l_{1},\alpha}(\varphi)
\end{equation}
Hence the result.
\end{proof}

The relationship between $\mathcal{W}_{\mathcal{U}}^{l,p}$ and $\mathcal{W}_{\mathcal{U}}^{-l,p}$ can be defined via the pairing
\begin{equation}
\label{eqn2.3}
\langle\mu,\varphi\rangle=\int(1+\vert\xi\vert^{p})^{l}\hat{\mu}(\xi)\bar{\hat{\varphi}}(\xi)d\xi=\int(1+\vert\xi\vert^{p})^{l} \vert\hat{\gamma}(\xi)\vert^{p}d\xi
\end{equation}
where $\displaystyle \hat{\mu}(\xi)\bar{\hat{\varphi}}(\xi)=\vert\hat{\gamma}(\xi)\vert$ for $\mu\in \mathcal{W}_{\mathcal{U}}^{l,p}$ and $\varphi\in \mathcal{W}_{\mathcal{U}}^{-l,p}$. We obtain the following
\[ \Big\vert\langle\mu,\varphi\rangle\Big\vert\leq \Big(\int(1+\vert\xi\vert^{p})^{l} \vert\hat{\mu}(\xi)\vert^{p}d\xi\Big)^{\frac{1}{p}}\Big(\int(1+\vert\xi\vert^{p})^{-l} \vert\hat{\varphi}(\xi)\vert^{p'}d\xi\Big)^{\frac{1}{p'}}=\Vert\mu\Vert_{l}\Vert\varphi\Vert_{-l} \]
for $\frac{1}{p}+\frac{1}{p'}=1$.

\begin{theorem}
\label{theo2.8}
There exists a canonical isometric isomorphism between $\mathcal{W}_{\mathcal{U}}^{-l,p}$ and $\Big(\mathcal{W}_{\mathcal{U}}^{l,p} \Big)'$ (dual of $\mathcal{W}_{\mathcal{U}}^{k,p}$) with the pairing (\ref{eqn2.3}).
\end{theorem}

\begin{proof}
Let $\displaystyle \mu\in \mathcal{W}^{-l,p}_{\mathcal{U}}$ be fixed such that $\mu:\mathcal{W}^{-l,p}_{\mathcal{U}}\longrightarrow\mathbb{C}$ defined by $\displaystyle \eta\mapsto\langle\mu,\eta\rangle=\mu(\eta)$ is a continuous linear functional on $\mathcal{W}^{l,p}_{\mathcal{U}}$ whose norm does not exceed the number $\Vert\cdot\Vert_{-l}$. Since $\mu$ is a tempered ultradistribution, it is evident to take $\eta_{0}=\mathcal{F}\Big[(1+\vert\xi\vert^{p})^{l}\hat{\mu}(\xi)\Big]\in \mathcal{W}_{\mathcal{U}}^{l,p}$. Thus we obtain that
\[ \langle \mu, \eta_{0}\rangle=\int (1+\vert\xi\vert^{p})^{l}\vert\hat{\mu}(\xi)\vert^{p}d\xi =\Vert\mu\Vert_{-l}\quad \text{for}\quad 1\leq p<\infty\].
Hence $\eta\mapsto \langle\mu,\eta\rangle=\Vert\mu\Vert_{-l}$ and isometry obtained.\ \\
Next we show that the isometry is surjective and hence isomorphism, let $\mu'\in\Big( \mathcal{W}_{\mathcal{U}}^{l,p}\Big)^{*}$.\\
Thus there is a  $\phi\in \mathcal{W}^{l,p}_{\mathcal{U}}$ such that
\begin{align*}
\mu'(\phi) &=\langle\mu',\phi\rangle_{l}=\langle\phi,\eta\rangle_{l}\\
      &=\int (1+\vert\xi\vert^{p})^{l}\hat{\phi}(\xi)\bar{\eta}(\nu)d\xi\\
      &=\int (1+\vert\xi\vert^{p})^{l}\vert\widehat{\Phi}(\xi)\vert^{p}d\xi\quad\forall\quad \eta\in \mathcal{W}_{\mathcal{U}}^{l,p}.
\end{align*}
This concludes the proof.
\end{proof}

\section{Some Structure Theorems and Embeddings}
In this section we consider some embedding results in the classical setting of Sobolev spaces involving tempered ultradistributions, and we extend the Rellich's compactness to the newly constructed classical Sobolev spaces since the underling functions satisfy both vanishing and growth conditions. In the framework of $\mathcal{W}^{l,p}_{\mathcal{U}}$, one can prove the $L^p$-property of large orders of derivatives with consideration given to the controlled growth of the defining functions. \ \\
We begin with following results'

\begin{theorem}
\label{eqn3.2}
Let $\mathfrak{K}= \lbrace z\in \mathbb{C}^{n}:\vert\text{Im}(z)\vert\leq l\rbrace$ for $l>0$, be an infinite strip, and let $\mu\in \mathcal{W}_{\mathcal{U}}^{l,p}$. Then $\mu$ is a holomorphic function in $\mathfrak{K}$.
\end{theorem}
\begin{proof}
Let $\displaystyle \mu(z)=(2\pi i)^{-n}\int(1+\vert\xi\vert^{p})^{l}\hat{\mu}(\xi)d\xi$ where $z$ is a complex number. Then for $\vert\beta\vert\leq l$ we have
\begin{align*}
(2\pi i)^{-n}\int(1+\vert\xi\vert^{p})^{l}\vert\hat{\mu}(\xi)\vert^{p}d\xi\leq \Vert\mu\Vert_{p}\Big(\int(1+\vert\xi\vert^{p})^{l}d\xi\Big)^{\frac{1}{p}}
\end{align*}
Since the integral of the polynomial is integrable if $\vert \beta\vert\leq l$, the result follows.
\end{proof}

\begin{definition}
\label{def3.3}
We define the topologies of $\displaystyle \mathcal{W}_{\mathcal{U}}^{\infty,p}=\bigcap_{l\in\mathbb{R}}\mathcal{W}_{\mathcal{U}}^{l,p}$ and $\displaystyle \mathcal{W}_{\mathcal{U}}^{-\infty,p}=\bigcup_{l\in\mathbb{R}}\mathcal{W}_{\mathcal{U}}^{l,p}$ as the weakest and strongest topology, respectively, such that the following injections $\mathcal{W}_{\mathcal{U}}^{\infty,p}\longrightarrow \mathcal{W}_{\mathcal{U}}^{l,p}$ and $vmathcal{W}_{\mathcal{U}}^{-\infty,p}\longrightarrow \mathcal{W}_{\mathcal{U}}^{l,p}$) are continuous.
\end{definition}

\begin{definition}
\label{def3.4}
Let $l\in \mathbb{Z}^{+}$. We define the space of entire functions, denoted  $\mathcal{H}^{l}(\Omega)$ with the uniform convergence for which the following condition holds:
\[\int_{\Omega}(1+\vert\xi\vert^{p})^{l}\vert\widehat{\eta\mu(\xi)}\vert d\xi<\infty\quad\text{for all}\quad \eta\in\mathcal{D}(\Omega)\]
\end{definition}
For the sake of the next result, let $\mathcal{K}'$ be space of finite order derivatives of some exponentially bounded continuous functions.

\begin{proposition}
\label{prop3.5}
Let $l$ be non-negative integer. If $\displaystyle\mathcal{U}'=\mathcal{F}\lbrace \mathcal{K}'\rbrace$. Then $\displaystyle \mathcal{K}^{l}\subset C^{l}(\Omega)$.
\end{proposition}

\begin{proof}
Given $\mu\in \mathcal{K}^{l}$ and $\xi\in \mathbb{C}$. Since $\mathcal{D}\subset \mathcal{K}$ then choose $\eta\in\mathcal{D}(\Omega)$ with $\eta=1$ within a neighborhood of $\xi$. Thus
\begin{equation}
\label{eqnU}
(\eta\mu)(\xi)=\int_{\Omega}(1+\vert\xi\vert^{p})^{l}\widehat{\eta\mu(\xi)} d\xi
\end{equation}
because $\eta\mu\in L^{1}(\mathbb{C})$. We have the inequality
\[(1+\vert\xi\vert^{p})^{\vert\beta\vert}\leq (1+\vert\xi\vert^{p})^{l}\]
for $\vert\beta\vert\leq l$.\\
Taking the derivative of the integral (\ref{eqnU}), we get

\begin{align*}
D^{\beta}(\eta\mu)(\xi) &= \int_{\Omega}D^{\beta}(1+\vert\xi\vert^{p})^{l}\hat{\eta\mu(\xi)} d\xi\\
          &=(-1)^{\vert\beta\vert}\int_{\Omega}(1+\vert\xi\vert^{p})^{l}D^{\beta}\hat{\eta\mu(\xi)} d\xi\\
          &<\infty
\end{align*}
Hence $\mu\in C^{l}(\mathbb{C})$.
\end{proof}
From Theorem \ref{thm2.7} we obtain the following result

\begin{theorem}
\label{th3.6}
Let $l>0$. Every function in $\mathcal{W}_{\mathcal{U}}^{-l,p}$ can be represented
\[\mu=\sum_{\beta\in\mathbb{N}_{0}^{n}}(-1)^{\vert\beta\vert}\frac{2\pi i}{\beta!}D^{\beta}\varphi_{\beta}\]
\end{theorem}
\begin{proof}
Let $\mu\in \mathcal{W}_{\mathcal{U}}^{-l,p}$ then by definition we have
\[(1+\vert\xi\vert^{p})^{l}\hat{\mu}(\xi)\in L^{2}(\mathbb{C})\]
which gives
\[\hat{\varphi}(\xi)=\frac{\hat{\mu}(\xi)}{\sum_{\beta}(-1)^{\vert\beta\vert}\frac{2\pi i}{\beta!}\vert\xi^{\beta}\vert}\in L^{2}\]
Thus
\begin{align*}
\hat{\mu}(\xi) &= \sum_{\vert\beta\vert\leq l}(-1)^{\vert\beta\vert}\frac{2\pi i}{\beta!}\vert\xi^{\beta}\vert \hat{\varphi}\\
            &=2\pi i \sum_{\vert\beta\vert\leq l}\frac{(-1)^{\vert\beta\vert}}{\beta!}\xi^{\beta}\Big(\frac{\vert\xi^{\beta}\vert}{\xi^{\beta}}\hat{\varphi}\Big)\\
             &=2\pi i \sum_{\vert\beta\vert\leq l}\frac{(-1)^{\vert\beta\vert}}{\beta!}\xi^{\beta}\hat{\varphi}_{\vert}
\end{align*}
Since $\hat{\varphi}_{\beta}(\xi)=\Big(\vert\xi^{\beta}\vert/\xi^{\beta}\Big)\hat{\varphi}(\xi)\in L^{2}(\mathbb{C})$, hence the proof.
\end{proof}

\begin{remark}
We see that Theorem \ref{th3.6} entails that any $\mu\in\mathcal{W}_{\mathcal{U}}^{l,p}$ can be represented as an infinite sum of partial derivatives of $L^2$-functions satisfying the growth conditions.
\end{remark}
We observed that if $\mu\in \mathcal{W}^{l,p}_{\mathcal{U}}$ then by definition \ref{def2.1}, $\hat{\mu}$ is also a tempered ultradistribution. From Plancherel's theorem, it implies that $\mathcal{W}_{\mathcal{U}}^{0,2}=L^{2}(\mathbb{C}^{n})$. That is, if $\mu\in L^{2}$ then $\hat{\mu}\in L^{2}$ and $\Vert \hat{\mu}\Vert_{2}=(2\pi i)^{\frac{n}{2}}\Vert \mu\Vert_{2}$.

We present the continuity result as well as the continuous embedding.
\begin{theorem}
\label{th6.9}
\begin{itemize}
\item[(i)] Let $t<l$ then the embedding $\displaystyle \mathcal{W}_{\mathcal{U}}^{l,p}\hookrightarrow \mathcal{W}_{\mathcal{U}}^{t,p}$ is continuous.
\item[(ii)] Let the linear partial differential operator with constant coefficient $\displaystyle P(D):\mathcal{W}_{\mathcal{U}}^{l,p}\hookrightarrow\mathcal{W}_{\mathcal{U}}^{l-\vert\beta\vert}$ of order $\vert\beta\vert\leq l$ be given. Then
$P(D)$ is continuous.
\end{itemize}
\end{theorem}

\begin{proof}
\begin{itemize}
\item[(i)] Let $t<l$, then we have
\begin{align*}
\Vert\mu\Vert_{\mathcal{W}_{\mathcal{U}}^{t,p}} &=(2\pi i)^{-\frac{n}{2}}\int (1+\vert\xi\vert^{p})^{t}\vert\hat{\mu}(\xi)\vert^{p}d\xi\\
      &=(2\pi i)^{-\frac{n}{2}}\int (1+\vert\xi\vert^{p})^{l}(1+\vert\xi\vert^{p})^{t-l}\Big\vert\hat{\mu}(\xi)\Big\vert d\xi\\
         &\leq \Vert\mu\Vert_{\mathcal{W}_{\mathcal{U}}^{l,p}}\quad\because (1+\vert\xi\vert^{p})^{t-l}\leq 1
\end{align*}
Hence the proof.
\item[(ii)] Let $P(D)=\partial^{\beta}$ and $\mu\in \mathcal{W}_{\mathcal{U}}^{l,p}$ be given such that by the exchange formula, we observe that
\begin{align*}
\Big(1+\vert\xi\vert^{\beta}\Big)^{l-\vert\beta\vert}\Big\vert \widehat{\partial^{\beta}\mu}\Big\vert  &\leq \Big(1+\vert\xi\vert^{p}\Big)^{l-\vert\beta\vert}\vert\xi\vert^{\vert\beta\vert}\vert\hat{\mu}(\xi)\vert \\
     &\leq   \Big(1+\vert\xi\vert^{p}\Big)^{l-\vert\beta\vert}\Big(1+\vert\xi\vert^{\beta}\Big)^{\vert\beta\vert} \vert\hat{\mu}(\xi)\vert         \\
     &=\Big(1+\vert\xi\vert^{p}\Big)^{l}\vert\hat{\mu}(\xi)\vert
\end{align*}
Whence
\[\Big(1+\vert\xi\vert^{\beta}\Big)^{l-\vert\beta\vert}\Big\vert \widehat{\partial^{\beta}\mu}\Big\vert  \leq\Big(1+\vert\xi\vert^{p}\Big)^{l}\vert\hat{\mu}(\xi)\vert\]
Therefore
\[\Vert \partial^{\beta}\mu  \Vert_{\mathcal{W}_{\mathcal{U}}^{l-\vert\beta\vert,p}} \leq \Vert \mu\Vert_{\mathcal{W}_{\mathcal{U}}^{l,p}}.\]
\end{itemize}
\end{proof}

\begin{remark}
\label{rem6.10}
The following inclusion arises from Theorem \ref{th6.9}(i):
\[\mathcal{U}\subset\mathcal{S}\subset \mathcal{W}_{\mathcal{U}}^{\infty,p}\subset \mathcal{W}_{\mathcal{U}}^{-\infty,p}\subseteq \mathcal{U}'\]
\end{remark}

We present the result below called the Sobolev norm of derivatives satisfying the growth conditions.

\begin{lemma}
\label{lem6.15}
Let $l\in\mathbb{R}$ and $p\in [1,\infty)$ be given. Then $\displaystyle \mu\in\mathcal{W}_{\mathcal{U}}^{l+1,p}$ if and only if the set $\displaystyle\Big\{(1+\vert\xi\vert^{p})^{\beta}\mu,(1+\vert\xi\vert^{p})^{\beta_1}\partial_{1}^{\beta_1}\mu,\cdots,(1+\vert\xi\vert^{p})^{\beta_n}\partial_{n}^{\beta_n}\mu\in \mathcal{W}_{\mathcal{U}}^{l,p}\Big\}$.
under the following norm
\[\Vert\mu\Vert_{\mathcal{W}_{\mathcal{U}}^{l+1,p}}^{p}=\Vert\mu\Vert_{\mathcal{W}_{\mathcal{U}}^{l,p}}^{p}+\sum_{i\in \mathbb{N}}\Vert\partial_{i}^{\beta}\mu\Vert_{\mathcal{W}_{\mathcal{U}}^{l,p}}^{p}.\]
\end{lemma}

\begin{proof}
Assume that $\displaystyle \mu\in\mathcal{W}_{\mathcal{U}}^{l+1,p}$ then $\displaystyle (1+\vert\xi\vert^{p})^{\beta}\mu\in\mathcal{W}_{\mathcal{U}}^{l,p}$ and $\displaystyle (1+\vert\xi\vert^{p})^{\beta}\partial_{i}^{\beta}\mu\in\mathcal{W}_{\mathcal{U}}^{l,p}$ for $i\in\mathbb{N}^n$. To see this, let $\displaystyle 1+\vert\xi\vert^{p}=1+\sum_{i\in\mathbb{N}}\xi^{p}$ which implies that $\displaystyle (1+\vert\xi\vert^{p})^{l}=\Bigg(1+\sum_{i\in\mathbb{N}}\xi^{p}\Bigg)^{l}$ . Thus for $\vert\beta\vert+1\leq l+1$
\begin{align*}
\Big\vert (1+\vert\xi\vert^{p})^{l+1}\hat{\mu}\Big\vert^{p}   &=\Big\vert(1+\vert\xi\vert^{p})\Big\vert^{p}\Big\vert (1+\vert\xi\vert^{p})^{l}\hat{\mu}\Big\vert^{p}\\
      &=   \Big\vert (1+\vert\xi\vert^{p})^{l}\hat{\mu}\Big\vert^{p}+\sum_{i=1}^{n}\Big\vert (1+\vert\xi\vert^{p})^{l}\xi_{i}\hat{\mu}\Big\vert^{p}\\
      &=\Big\vert (1+\vert\xi\vert^{p})^{l}\hat{\mu}\Big\vert^{p}+\sum_{i=1}^{n}\Big\vert (1+\vert\xi\vert^{p})^{l}\widehat{\partial_{i}^{\beta}\mu}\Big\vert^{p}
\end{align*}
as required for $\vert\beta\vert\leq l$, and since $\sup_{\xi\in\mathbb{C}^{n}}\Big\{ (1+\vert\xi\vert^{p})^{l}\vert\partial_{i}^{\beta}\mu(\xi)\vert\Big\}<\infty$, it implies that $\mu, \partial_{i}^{\beta}\mu$ belong to $\mathcal{W}_{\mathcal{U}}^{l+1,p}$.
\end{proof}

We have the characterization of the space $\mathcal{W}_{\mathcal{U}}^{\vert\beta\vert,p}$ Theorem \ref{th6.16} below.

\begin{theorem}
\label{th6.16}
Let $\beta\in\mathbb{N}^{n}$ be a multi-index such that $\vert\beta\vert=\beta_{1}+\cdots+\beta_{n}$. Then we have
\begin{equation}
\label{eqnth6.16}
\mathcal{W}^{\vert\beta\vert}_{\mathcal{U}}(\mathbb{C}^{n})\equiv\Big\{\mu\in \mathcal{U}'(\mathbb{C}^{n}):\partial^{\beta}\mu\in L^{p}\quad\forall\quad\vert\beta\vert\leq l\Big\}
\end{equation}
\end{theorem}

\begin{proof}
From definition of $\mathcal{W}_{\mathcal{U}}^{l,p}$, we show that $\mathcal{W}_{\mathcal{U}}^{l,p}$ can be expressed as $\mathcal{W}_{\mathcal{U}}^{\vert\beta\vert,p}$. To do this, let $\mu\in \mathcal{U}'$ then we have that $\displaystyle (1+\vert\xi\vert^{p})^{l}\mu\in \mathcal{W}^{\vert\beta\vert}_{\mathcal{U}}(\mathbb{C}^{n})$ and this shows that $\partial^{\vert\beta\vert}\mu\in\mathcal{W}^{\vert\beta\vert}_{\mathcal{U}}(\mathbb{C}^{n})$ and $(1+\vert\xi\vert^{p})\partial^{\vert\beta\vert}\mu\in\mathcal{W}^{\vert\beta\vert}_{\mathcal{U}}(\mathbb{C}^{n})$. \ \\
Thus
\begin{align*}
\Bigg(1+\vert\xi\vert^{p}\Bigg)^{(\beta_{1},\cdots,\beta_{n})}\partial_{i}^{\beta_{1}+\cdots+\beta_n}\mu=(1+\vert\xi\vert^{p})^{\beta_{1}}\partial^{\beta_1}_{i}\mu\cdots (1+\vert\xi\vert^{p})^{\beta_{n}}\partial^{\beta_n}_{i}\mu\in \mathcal{W}_{\mathcal{U}}^{\vert\beta\vert,p}
\end{align*}
provided $\vert\beta\vert\leq l$.
\end{proof}

consider the following embedding result.

\begin{theorem}
\label{th7.21}
\begin{itemize}
\item[(i)] If $n>2\vert\beta\vert$ for $l<\frac{n}{2}$  then $\displaystyle \mathcal{W}_{\mathcal{U}}^{l,p}\subseteq C^{0}_{0}(\mathbb{C}^{n})$
\item[(ii)] Let $\vert\beta\vert<r+n/2$, $r\in\mathbb{R}$ for $l<r+\frac{n}{2}$ and $\vert\beta\vert\leq l$.  Then $\displaystyle \mathcal{W}_{\mathcal{U}}^{l,p}\subseteq C_{0}^{l}(\mathbb{C}^{n})$.
\end{itemize}
\end{theorem}

\begin{proof}
\begin{itemize}
\item[(i)] Assume that $\mu\in \mathcal{W}_{\mathcal{U}}^{l,p}$ with $n<2\vert\beta\vert$. From definition of $\mathcal{U}'$, then we have that $\displaystyle \xi\mapsto (1+\vert\xi\vert^{p})^{-l}\xi\in L^{p-1}$ for $p\geq 2$ and therefore $\eta=(1+\vert\xi\vert^{p})^{l}\hat{\mu}$. An application of Cauchy-Schwartz inequality yields
\begin{align*}
\Vert\hat{\mu}\Vert_{L^{p-1}} & \leq \Vert  \eta\Vert_{L^p}\Big(\int(1+\vert\xi\vert^{p})^{\vert\beta\vert}d\xi\Big)^{\frac{1}{p}}\\
        & \leq \Vert  \eta\Vert_{L^p}\Big(\sup_{\xi\in\mathbb{C}}\int(1+\vert\xi\vert^{p})^{l}d\xi\Big)^{\frac{1}{p}}\\
        & \leq M\Vert \eta\Vert_{L^p}\leq M\Vert \eta\Vert_{\mathcal{W}_{\mathcal{U}}^{l,p}}
\end{align*}
This shows that $\displaystyle \mathcal{W}_{\mathcal{U}}^{l,p}\hookrightarrow L^{p}\hookrightarrow L^{p-1}$. Thus  $\hat{\mu}\in L^{p-1}(\mathbb{C}^{n})$ and $\mu=\mathcal{F}^{-1}(\hat{\mu})$ be classical inversion formula. By Riemann-Lebesgue lemma, it implies that $\mu \in C^{0}_{0}$.
\item[(ii)] From definition of $C_{0}^{l}$, we have $f\in C_{0}^{l}$ such that $\displaystyle \lim_{\vert\xi\vert\to\infty}f(\xi)=0$ $\forall\quad \vert\beta\vert\leq r$.  Assume that $\vert\beta\vert<r+\frac{n}{2}$ and let $\mu\in \mathcal{W}_{\mathcal{U}}^{l,p}$. Then $\partial^{\beta}\mu\in \mathcal{W}^{l-r}(\mathbb{C}^{n})$ for all $\vert\beta\vert\leq r$ then we see that $\partial^{\beta}\mu\in C^{\infty}(\mathbb{C}^n)\quad\forall\quad\beta$. Thus $\partial^{\beta}\mu\in \mathcal{W}_{\mathcal{U}}^{l-r,p}\implies \partial^{\beta}\mu\in \mathcal{W}_{\mathcal{U}}^{l,p}$.
\end{itemize}
\end{proof}

\begin{corollary}
\label{cor6.22}
If $\mu\in \mathcal{W}_{\mathcal{U}}^{\infty,p}(\mathbb{C}^{n})$ then $\mu\in C_{0}^{\infty}$.
\end{corollary}

\begin{proof}
We show that $\mu\in C_{0}^{\infty}$ which implies that $\mu\in C^{\infty}\cap C_{0}$ such that $\text{supp}\mu$ has a compact support in $K=\bigcup K_{i}$ provided $\mu\in \mathcal{W}_{\mathcal{U}}^{\infty,p}$. To show this, let $\mu\in \mathcal{W}_{\mathcal{U}}^{\infty,p}$ then $\partial^{\beta}\mu\in \mathcal{W}_{\mathcal{U}}^{\infty,p}$ for $\vert\beta\vert\leq \infty$ within $K$ as neighborhood of zero. Take $\eta\in L^{\infty}$, then $\mu\eta\in L^{\infty}(\Omega)$. Thus
\[(\eta\mu)(y)=\int_{K}(1+\vert\xi\vert^{p})^{l}\widehat{\eta\mu}(\xi)d\xi\]
Using the inequality $\vert\beta\vert\leq l$ and $\vert\beta\vert\leq \infty$,
\[\vert\xi\vert^{\beta}\leq (1+\vert\xi\vert^{p})^{\vert\beta\vert}\leq(1+\vert\xi\vert^{p})^{l}\leq(1+\vert\xi\vert^{p})^{l}\widehat{\eta\mu}(\xi)<\infty\]
has the smallest value for which $\vert\beta\vert\leq \infty$ since $\eta\mu\in L^{\infty}$ has the common tempered ultradistribution in $\mathcal{W}_{\mathcal{U}}^{\infty,p}$ since it satisfies some growth condition. \ \\
Alternatively, let $\mu\in\mathcal{W}_{\mathcal{U}}^{\infty,p}$ then
\[\vert\hat{\mu}(\xi)\vert\leq (1+\vert\xi\vert^{p})^{l}\vert\hat{\mu}(\xi)\vert,\quad\xi\in\mathbb{C}^{n}\]
Hence
\begin{align*}
\int_{K}(1+\vert\xi\vert^{p})^{l}\vert\hat{\mu}(\xi)\vert^{p}d\xi & \leq M \int_{K}(1+\vert\xi\vert^{p})^{l}d\xi\\
     & \leq M \int_{K}(1+\vert\xi\vert^{p})^{l+r}d\xi
\end{align*}
for $l+r<(-n/2)$. Hence the result.
\end{proof}

\begin{remark}
The space $\mathcal{W}_{\mathcal{U}}^{l,p}$ is closed under multiplication of functions with rapid decrease. One can see that if $\mu\in \mathcal{W}^{l,p}_{\mathcal{U}}$ and $\phi\in \mathcal{U}$ then $\phi\mu\in \mathcal{W}_{\mathcal{U}}^{l,p}(\mathbb{C}^n)$ with the map that $\phi\mapsto\phi\mu$ is continuous on $\mathcal{W}_{\mathcal{U}}^{l,p}$ since $\displaystyle \sup_{\xi\in \mathbb{C}^{n}}\Big((1+\vert\xi\vert^{p})^{l}\vert \phi\mu(\xi)\vert^{p}\Big)<\infty$.
\end{remark}

We present the following embedding theorem

\begin{proposition}
\label{prop3.3}
The space $\mathcal{U}'(\mathbb{C}^{n})$ is dense in $\mathcal{W}^{-\infty,p}_{\mathcal{U}}(\mathbb{C}^{n})$. That is, $\displaystyle \overline{\mathcal{U}'}= \mathcal{W}^{-\infty,p}_{\mathcal{U}}$.
\end{proposition}

\begin{proof}
We show that $\displaystyle \overline{\mathcal{U}'}= \mathcal{W}^{-\infty,p}_{\mathcal{U}}$. Recall that $\displaystyle\mathcal{W}^{-\infty,p}_{\mathcal{U}}=\bigcup_{l\in\mathbb{R}}\mathcal{W}_{\mathcal{U}}^{l,p}$. To do this let $\Big\{(1+\vert\xi\vert^{p})^{l}\mu(\xi)\Big\}$ denotes a countable convergent sequence of rapidly decreasing ultradifferentiable functions $\mu$ satisfying the growth condition such that $\displaystyle\sup_{\xi\in\mathbb{C}}\Big\{ (1+\vert\xi\vert^{p})^{l}\mu(\xi)\Big\}<\infty$. Then it follows that $\displaystyle \vert\hat{\mu}(\xi)\vert\leq M (1+\vert\xi\vert^{p})^{l}\vert \eta\mu(\xi)\vert^{p}<\infty$ for all $\eta\in\mathcal{U}$. Thus it follows that $\displaystyle\lim_{\vert\xi\vert\to\infty} \Big\{ (1+\vert\xi\vert^{p})^{l}\mu(\xi)\Big\}=0$ for $\vert\beta\vert\leq l$. \ \\
Therefore
\[\int_{\Omega}(1+\vert\xi\vert^{p})^{l}\vert \eta\widehat{\mu}(\xi)\vert^{p}d\xi\leq M \int_{\Omega}(1+\vert\xi\vert^{p})^{l-\vert\beta\vert}d\xi<\infty\quad\text{for}\quad \vert\beta\vert\leq l\quad\text{if}\quad l-\vert\beta\vert<-(n/2)\]
Then it shows that $\mu$ belongs to $\mathcal{W}_{\mathcal{U}}^{l,p}$ and $\cup\mathcal{W}_{\mathcal{U}}^{l,p}$.
\end{proof}

\begin{corollary}
\label{cor3.4}
If $\phi\in\mathcal{U}$ then the multiplication operator operator, $T_{\phi}:\mathcal{U}\longrightarrow \mathcal{W}_{\mathcal{U}}^{l,p}$ defined by $T_{\phi}\mu=\phi\mu$ is a bounded linear operator on $\mathcal{W}_{\mathcal{U}}^{l,p}$ for all $l\in \mathbb{R}$.
\end{corollary}

\begin{definition}
\label{def3.5}
We denote denote by $\overline{\mathcal{W}}_{\mathcal{U}}^{l,p}$ the closure of $\mathcal{D}$ in the $\mathcal{W}_{\mathcal{U}}^{l,p}$-norm for all $l\in\mathbb{R}$.
\end{definition}
Before we present the extension of the Rellich's compactness theorem, the trace theorem is needed for the proof. The importance of the Rellich's theorem establish the fact that under certain conditions the embedding $\mathcal{W}^{l,p}_{\mathcal{U}}\hookrightarrow \mathcal{W}^{t,p}_{\mathcal{U}}$ for $l>t$ is compact. This means that from any bounded sequence in $\mathcal{W}_{\mathcal{U}}^{l,p}$ one can obtain an $\mathcal{W}^{t,p}_{\mathcal{U}}$-convergent subsequence. The elementary result called Bolzano-Weierstrass theorem plays a significant role in this compactness result.

\begin{theorem}
\label{th3.5}
If $\mu\in \overline{\mathcal{W}}_{\mathcal{U}}^{l,p}$ for some $l>n/2$, then $\mu\equiv 0$ on $\partial\Omega$ for $\Omega\subset \mathbb{C}^n$.
\end{theorem}

\begin{proof}
Let $l>n/2$ and $\mu\in \overline{\mathcal{W}}_{\mathcal{U}}^{l,p}$ be given, then we have a convergent sequence $\lbrace\mu_{m}\rbrace$ of functions in $\mathcal{D}(\Omega)$ with compact support since $\mathcal{D}\subset\mathcal{U}\subset\mathcal{S}$ such that $\mu_m\longrightarrow m$ in $\mathcal{W}_{\mathcal{U}}^{l,p}$.  By Theorem \ref{th7.21}, we get
\[\sup_{\xi\in\partial\Omega}\Big\vert(1+\vert\xi\vert^{p})^{l}\mu(\xi)\Big\vert=\sup_{\xi\in\partial\Omega}\Big\vert(1+\vert\xi\vert^{p})^{l}(\mu-\mu_{m})(\xi)\Big\vert\leq  M_{l}\Vert\mu_{m}-\mu\Vert_{l}\quad\text{for}\quad m=2,\cdots,\]
Since $\displaystyle \Vert\mu_{m}-\mu\Vert$ is non-negative, it implies that the equality occurs when $\mu \equiv 0$ on $\partial\Omega$.
\end{proof}

\begin{remark}
Theorem \ref{th3.5} entails that function $\mu$ with slow growth condition vanishes at infinity on the boundary of the open subset. \ \\
Also from Theorem \ref{th3.5}, we characterize the trace of a Sobolev function $\mu\in \mathcal{W}^{1,p}_{\mathcal{U}}$ namely, the ``restriction`` of $\mu$ to the boundary $\partial\Omega$.
\end{remark}

\begin{theorem}
\label{th3.6}
Let $\Omega\subset \mathbb{C}^{n}$ be a bounded open subset, and let $l<l'$ be given. Then there exists a convergent subsequence in $\mathcal{W}_{\mathcal{U}}^{l,p}$ for every bounded sequence of tempered ultradistributions in $\overline{\mathcal{W}}_{\mathcal{U}}^{l',p}(\Omega)$. In particular, the inclusion map $\displaystyle\overline{\mathcal{W}}_{\mathcal{U}}^{l',p}(\Omega)\longrightarrow \mathcal{W}_{\mathcal{U}}^{l,p}$ is compact.
\end{theorem}

\begin{proof}
First we show that both sequences of rapidly decreasing ultradifferentiable functions and its partial derivatives defining tempered ultradistribtions is bounded. To do this, let $\lbrace\mu_{m}\rbrace$ be a bounded sequence in $\overline{\mathcal{W}}_{\mathcal{U}}^{l',p}(\Omega)$ and let $\chi$ be a local unit for $\bar{\Omega}$. Then we have $\vert\beta\vert\leq l'$ and $\vert\beta\vert\leq l$,
\begin{align*}
(1+\vert\xi\vert^{p})^{l'}\vert\hat{\mu}_{m}(\xi)\vert & \leq \int_{\Omega}(1+\vert\xi\vert^{p})^{l'-l}  \vert\hat{\chi}(\xi)\vert (1+\vert\xi\vert^{p})^{l}\vert\hat{\mu}_{m}(\xi)\vert d\xi \\
     &\leq \Bigg[\int_{\Omega}(1+\vert\xi\vert^{p})^{l'-l}  \vert\hat{\chi}(\xi)\vert^{p}d\xi\Bigg]^{1/p}   \Bigg[\int_{\Omega}(1+\vert\xi\vert^{p})^{l}  \vert\hat{\mu}_{m}(\xi)\vert^{p} d\xi \Bigg]^{1/p}\\
     &\leq \Vert\mu_{m}\Vert_{l}\Bigg[\int_{\Omega}(1+\vert\xi\vert^{p})^{l'-l}\vert\hat{\chi}(\xi)\vert^{p}d\xi\Bigg]^{1/p} \\
     &\leq \Vert\mu_{m}\Vert_{l}\sup_{\xi\in\Omega}\Big\{(1+\vert\xi\vert^{p})^{l'-l} \vert\hat{\chi}(\xi)\vert^{p}\Big\}\\
     & \leq M_{1}<\infty.
\end{align*}
for $M_{1}>0$ (not dependent on $\xi$ and $m$).\\
We also prove, in a similar way, the boundedness of the sequence of partial derivatives of tempered ultradistributions. Whence,
\begin{align*}
(1+\vert\xi\vert^{p})^{l'}\vert\partial_{i}^{\beta}\hat{\mu}_{m}(\xi)\vert & \leq \int_{\Omega}(1+\vert\xi\vert^{p})^{l'-l}  \vert\partial^{\beta}\hat{\chi}(\xi)\vert (1+\vert\xi\vert^{p})^{l}\vert\partial^{\beta}\hat{\mu}_{m}(\xi)\vert d\xi \\
     &\leq \Bigg[ \int_{\Omega}\partial^{\beta}(1+\vert\xi\vert^{p})^{l'-l}  \vert\hat{\chi}(\xi)\vert^{p}d\xi\Bigg]^{1/p}   \Bigg[\int_{\Omega}\partial^{\beta}(1+\vert\xi\vert^{p})^{l}  \vert\hat{\mu}_{m}(\xi)\vert^{p} d\xi \Bigg]^{1/p}\\
     &\leq \Vert\mu_{m}\Vert_{l}\vert\xi\vert^{p}\Bigg[\int_{\Omega}(1+\vert\xi\vert^{p})^{l'-l}  \vert\partial^{\beta}\hat{\chi}(\xi)\vert^{p}d\xi\Bigg]^{1/p} \\
     &\leq \Vert\mu_{m}\Vert_{l}\sup_{\xi\in\Omega}\Bigg\{(1+\vert\xi\vert^{p})^{l'-l}\vert\chi\vert^{\beta} \vert\hat{\chi}(\xi)\vert^{p}\Bigg\}\\
     & \leq M_{2}<\infty.
     \end{align*}
Next we show that $(1+\vert\xi\vert^{p})^{\beta}\hat{\mu}_{m}(\xi)$ is an equicontinuous sequence of rapidly decreasing ultradifferentiable functions on compact set. To do this, let $K$ be any compact subset of $\mathbb{C}^n$ and $\displaystyle  M_{3}=\sup_{\xi\in\mathbb{C}}\Bigg\{(1+\vert\xi\vert^{p})^{\beta}: \xi\in \overline{B(0,T)}\Bigg\}$ where $\displaystyle T=\sup_{\xi\in K}\Vert\xi\Vert$. Then $\displaystyle \vert\hat{\mu}_{m}\vert\leq M_{1}(1+\vert\xi\vert^{p})^{-l}\leq M_{2}M_{3}$ and $\displaystyle \vert\partial^{\beta}\hat{\mu}_{m}\vert\leq M_{2}M_{3}$ for every $m>0$ and $\xi$ in $K$, $\vert\beta\vert\leq l$. Due to the uniform continuity nature of polynomials $(1+\vert\xi\vert^{p})^{l}$ on $K$, then for every $\varepsilon>0$, constant $\delta>0$, we have that
\[\Big\vert (1+\vert\xi\vert^{p})^{l}-(1+\vert\eta\vert^{p})^{l} \Big\vert<\frac{\varepsilon}{2M_{1}M_{3}}\quad\text{whenever}\quad \vert\xi-\eta\vert<\delta_{1}\]
Then from mean value theorem, we get
\begin{align*}
\Big\vert\hat{\mu}_{m}(\xi)-\hat{\mu}_{m}(\eta)\Big\vert &\leq \sum_{i\in\mathbb{N}}\Big\vert\xi_{i}-\eta_{i}\Big\vert^{p}\Big\vert\partial_{i}^{\beta}\hat{\mu}_{m}(\xi-\eta)\Big\vert\\
    &\leq M_{4}\Vert\xi-\eta\Vert
    \end{align*}
for some constant $M_4$, $m\in\mathbb{N}$ and $\xi,\eta\in K$.\\
Also, for every $\varepsilon>0$ there is a constant $\delta_{2}>0$ such that $\displaystyle \Big\vert\hat{\mu}_{m}(\xi)-\hat{\mu}_{m}(\eta)\Big\vert<\frac{\varepsilon}{2M_{3}M_{4}}$ whenever $\vert\xi-\eta\vert\leq \delta_{2}$. Now choose $\delta=\min\lbrace\delta_{1},\delta_{2}\rbrace$, we get
\begin{align*}
\Big\vert(1+\vert\xi\vert^{p})^{l}\hat{\mu}_{m}(\xi)-(1+\vert\eta\vert^{p})^{l}\hat{\mu}_{m}(\eta)\Big\vert &\leq (1+\vert\xi\vert^{p})^{l} \Big\vert\hat{\mu}_{m}(\xi)-\hat{\mu}_{m}(\eta)\Big\vert\\
    &+ \Big\vert(1+\vert\xi\vert^{p})^{l}-(1+\vert\eta\vert^{p})^{l}\Big\vert\Big\vert\hat{\mu}(\eta)\Big\vert\\
    &<\varepsilon
    \end{align*}
 provided $\vert\xi-\eta\vert<\delta$ and $\xi,\eta\in K$ which established the uniform continuity of $(1+\vert\xi\vert^{p})^{l}$.\\
 This shows that $\displaystyle \Big\{ (1+\vert\xi\vert^{p})^{l}\hat{\mu}_{m}\Big\}$ is a pointwise bounded sequence in $\overline{W}_{\mathcal{U}}^{k,p}$ on $K$. The convergent subsequence $\displaystyle \Big\{ (1+\vert\xi\vert^{p})^{\beta}\hat{\mu}_{m_{i}}\Big\}$ by Arzela-Ascoli theorem  and Bolzano-Weierstrass theorem  exists uniformly on each compact sets in $\mathbb{C}^n$. \\
For every $T>0$, we have
\begin{align}
\label{eqn7.1}
\Vert\mu_{m_i}-\mu_{m_j}\Vert &=\int_{\Omega}(1+\vert\xi\vert^{p})^{2l}\vert(\mu_{m_i}-\mu_{m_j})(\xi)\vert^{p}d\xi\nonumber \\
        &=\int_{\Omega}(1+\vert\xi\vert^{p})^{2l}(1+\vert\xi\vert^{p})^{2l-2k}\vert(\mu_{m_i}-\mu_{m_j})(\xi)\vert^{p}d\xi\nonumber \\
        & \leq
        \int_{\vert\xi\vert\geq T}(1+\vert\xi\vert^{p})^{2l-2k}\vert(\mu_{m_i}-\mu_{m_j})(\xi)\vert^{p}d\xi\nonumber\\
        &+\int_{\vert\xi\vert\leq T}(1+\vert\xi\vert^{p})^{2k}\vert(\mu_{m_i}-\mu_{m_j})(\xi)\vert^{p}d\xi\nonumber\\
       &\leq \int_{\vert\xi\vert\leq T}(1+\vert\xi\vert^{p})^{2k}\vert(\mu_{m_i}-\mu_{m_j})(\xi)\vert^{p}d\xi+M
\end{align}
The inequality of (\ref{eqn7.1}) follows from the boundedness of $\lbrace\mu_{m}\rbrace$ in $\overline{\mathcal{W}}_{\mathcal{U}}^{l,p}$. Now given $\varepsilon>0$, we can take a large $T$ so that due to the uniform convergence on $B(0,T)$ of the subsequence $\displaystyle (1+\vert\xi\vert^{p})^{l}\hat{\mu}_{m_j}(\xi)$ then given a constant such that $m_{i},M_{j}\geq M$ for $i<j$ we have
\[\int_{\vert\xi\vert\leq T}(1+\vert\xi\vert^{p})^{l}\vert\mu_{m_i}-\mu_{m_j}\vert^{p}d\xi<\frac{\varepsilon^p}{2}+\frac{\varepsilon}{2}<\varepsilon\]
Hence $\lbrace\mu_{j}\rbrace$ is a convergent subsequence of $\mu_m$ in $W_{\mathcal{U}}^{t,p}$.
\end{proof}
\begin{remark}
If $l=t$, then Theorem \ref{th3.6} will not hold since their orders of derivatives are same.
\end{remark}

\begin{corollary}
\label{corembedding}
Let $\Omega\subset\mathbb{C}$ be a bounded subset and $l<t<t'$. Then we have the following continuous embedding
\begin{equation}
\label{eqnembedding}
\mathcal{U}\hookrightarrow \mathcal{W}^{t',p}_{\mathcal{U}}\hookrightarrow\mathcal{W}_{\mathcal{U}}^{t,p}\hookrightarrow\mathcal{W}_{\mathcal{U}}^{l,p}\hookrightarrow\mathcal{U}'
\end{equation}
\end{corollary}

\begin{proof}
There is need to show that $\displaystyle \mathcal{W}_{\mathcal{U}}^{l,p}\hookrightarrow\mathcal{U}'$ since other inclusions are obvious. To see this, let $\mu\in\mathcal{W}_{\mathcal{U}}^{l,p}$ such that $\displaystyle \sup_{\xi\in\mathbb{C}}\Big\{ (1+\vert\xi\vert^{p})^{l}\vert\mu(\xi)\vert\Big\}$ is finite for $\vert\beta\vert\leq l$ and $\displaystyle \sup_{\xi\in\mathbb{C}}\Big\{ (1+\vert\xi\vert^{p})^{l}\partial^{\beta}\mu(\xi)\Big\}<\infty$. \\
Whence
\begin{align*}
(1+\vert\xi\vert^{p})^{l}\vert\hat{\mu}(\xi)\vert  & \leq \Bigg[\int_{\Omega}(1+\vert\xi\vert^{p})^{\vert\beta\vert}\vert\hat{\mu}(\xi)\vert^{p}d\xi\Bigg]^{\frac{1}{p}}\\
     & \leq \sup_{\xi\in\mathbb{C}} \Bigg[\int_{\Omega}(1+\vert\xi\vert^{p})^{l}\vert\hat{\mu}(\xi)\vert^{p}d\xi\Bigg]^{\frac{1}{p}}\\
     &   \leq \Vert\mu\Vert_{\mathcal{W}_{\mathcal{U}}^{l,p}}
\end{align*}
as required.
\end{proof}



\textbf{Conflict of Interest}: The authors demonstrated no conflict of interest towards the publication of this manuscript.

{}

\end{document}